\documentclass[11pt,reqno]{amsart}
\usepackage{amsmath,amsopn,amssymb,amsthm,multicol, bbm}
\usepackage{hyperref}
\usepackage{graphicx}
\usepackage{color}

\DeclareMathOperator{\Ad}{Ad}

\DeclareMathOperator{\Aut}{Aut}

\DeclareMathOperator{\Id}{Id}
\newcommand{\op}{\operatorname}

\DeclareMathOperator{\Span}{span}

\newcommand{\fr}{\mathfrak}

\newcommand{\al}{\alpha}

\newcommand{\bb}{\mathbb}

\DeclareMathOperator{\SO}{SO}

\DeclareMathOperator{\OO}{O}
\DeclareMathOperator{\Sp}{Sp}

\DeclareMathOperator{\U}{U}

 \newtheorem{lemma} {Lemma} [section]
\newtheorem{theorem}[lemma]{Theorem} 
\newtheorem{remark}[lemma] {Remark} 
\newtheorem{prop} [lemma]{Proposition}  
\newtheorem{definition}[lemma] {Definition} 
\newtheorem{corol}[lemma] {Corollary}

\hypersetup{
  colorlinks   = true, 
  urlcolor     = black, 
  linkcolor    = red, 
  citecolor   = blue 
}

\begin{document}

\title[Geodesic orbit metrics in a class of homogeneous bundles]{Geodesic orbit metrics in a class of homogeneous bundles over real and complex Stiefel manifolds}

\author{Andreas Arvanitoyeorgos$^\ast$, Nikolaos Panagiotis Souris and Marina Statha}
\address{University of Patras, Department of Mathematics, GR-26500 Rion, Greece}
\email{arvanito@math.upatras.gr}
\address{University of Patras, Department of Mathematics, GR-26500 Rion, Greece}
\email{nsouris@upatras.gr  }
\address{University of Patras, Department of Mathematics, GR-26500 Rion and
University of Thessaly,  Department of Mathematics, GR-35100 Lamia, Greece}
\email{statha@math.upatras.gr}

\begin{abstract}
Geodesic orbit spaces (or g.o. spaces) are defined as those homogeneous Riemannian spaces $(M=G/H,g)$ whose geodesics are orbits of one-parameter subgroups of $G$.  The corresponding metric $g$ is called a geodesic orbit metric.  We study the geodesic orbit spaces of the form $(G/H,g)$, such that $G$ is one of the compact classical Lie groups $\SO(n)$, $U(n)$, and $H$ is a diagonally embedded product $H_1\times \cdots \times H_s$, where $H_j$ is of the same type as $G$. This class includes spheres, Stiefel manifolds, Grassmann manifolds and real flag manifolds.  The present work is a contribution to the study of g.o. spaces $(G/H,g)$ with $H$ semisimple.
    

\end{abstract}

\maketitle

\section{Introduction}

Geodesic orbit spaces $(M=G/H,g)$ are defined by the simple property that any geodesic $\gamma$ has the form 
\begin{equation*}\gamma(t)=\exp(tX)\cdot o,\end{equation*}
 where $\exp$ is the exponential map on $G$, $o=\gamma(0)$ is a point in $M$ and $\cdot$ denotes the action of $G$ on $M$.  These spaces were initially considered in \cite{Ko-Va}, and up to this day they have been extensively studied within various geometric contexts, including the Riemannian (\cite{GoNi}), pseudo-Riemannian (\cite{CaZa}), Finsler (\cite{YaDe}) and affine (\cite{Du}) context.  The classification of g.o. spaces remains an open problem, whereas several partial classifications have been obtained (\cite{Al-Ar}, \cite{Al-Ni}, \cite{CheCheWo}, \cite{CheNi}, \cite{CheNiNi}, \cite{Go}, \cite{So2}, \cite{Ta} to name a few).

There are diverse examples of g.o. spaces, including the classes of symmetric spaces, weakly symmetric spaces (\cite{Ber-Ko-Va}, \cite{Wo2}), isotropy irreducible spaces (\cite{Wo1}), $\delta$-homogeneous spaces (\cite{Be-Ni-1}) and Clifford-Wolf homogeneous spaces (\cite{Be-Ni-2}).  The most important subclass of g.o. spaces are the \emph{naturally reductive spaces}, whose complete description is also open (see the recent low-dimensional classifications \cite{Ag}, \cite{St}).  Another related topic of recent interest is the study of Einstein metrics that are not g.o. metrics (\cite{CheCheDe}, \cite{Ni3}). 
 For a review about g.o. spaces we refer to \cite{Ar2} and in the introduction of the article \cite{Ni2}. We also point out the recently published book \cite{Be-Ni-3}.


Determining the g.o. metrics among the $G$-invariant metrics on a space $G/H$ presents some challenges.  The main challenge lies in the fact that the space of $G$-invariant metrics may have complicated structure, depending on whether the \emph{isotropy representation} of $H$ on the tangent space $T_{o}(G/H)$ contains pairwise equivalent submodules. To remedy this obstruction, various simplification results for g.o. metrics have been established (e.g. \cite{Ni2}, \cite{So1}).  A general observation is that the existence and the form of the g.o. metrics on $G/H$ depends to a large extent on the structure of the tangent space $T_{o}(G/H)$ induced from the isotropy representation and on the Lie algebraic relations between the corresponding submodules (e.g. \cite{CheNiNi}).

 When $G$ is compact semisimple, the classification of the g.o. spaces $(G/H,g)$ with $H$ abelian and $H$ simple has been obtained in the works \cite{So2} and \cite{CheNiNi} respectively. On the other hand, the classification of compact g.o. spaces $(G/H,g)$ with $H$ semisimple remains open, while no general results are known for this case.  As a first step towards this direction, in this paper we study the g.o. metrics on a general family of spaces $G/H$ with $H$ semisimple, such that the isotropy representation of all of its members has a similar description.

 In particular, we consider spaces $G/H$ where $G$ is a compact classical Lie group and $H$ is a diagonally embedded product of Lie groups of the same type as $G$.  More specifically, we study the spaces $\SO(n)/\SO(n_1)\times \cdots \times \SO(n_s)$ and $\U(n)/\U(n_1)\times \cdots \times \U(n_s)$ with $0<n_1+\cdots +n_s\leq n$.  This class properly includes the spheres $\SO(n)/\SO(n-1)$ and $\U(n)/\U(n-1)$, the Stiefel manifolds $\SO(n)/\SO(n-k)$ and $\U(n)/\U(n-k)$, the Grassmann manifolds $\SO(n)/\SO(k)\times \SO(n-k)$, $\U(n)/\U(k)\times \U(n-k)$ as well as the real flag manifolds $\SO(n)/\SO(n_1)\times \cdots \times \SO(n_s)$ and $\U(n)/\U(n_1)\times \cdots \times \U(n_s)$ with $n_1+\cdots +n_s=n$. If $n_1+\cdots +n_s< n$, each of these spaces can be viewed as a total space over a Stiefel manifold, with the fiber being a real flag manifold, e.g.
\begin{equation*}
 \SO(m)/\SO(n_1)\times \cdots \times \SO(n_s)\rightarrow \SO(n)/\SO(n_1)\times \cdots \times \SO(n_s)\rightarrow \SO(n)/\SO(m),
\end{equation*}
 with $m=n_1+\cdots +n_s$. The first main result is the following.

\begin{theorem}\label{main1}
Let $G/H$ be the space $\SO(n)/\SO(n_1)\times \cdots \times \SO(n_s)$, where $0<n_1+\cdots +n_s\leq n$, and $n_j>1, j=1,\dots , s$.  A $G$-invariant Riemannian metric on $G/H$ is geodesic orbit if and only if it is a normal metric, i.e. it is induced from an $\op{Ad}$-invariant inner product on the Lie algebra $\fr{so}(n)$ of $\SO(n)$.
\end{theorem}

\begin{remark} {\rm We remark that if $n\neq 4$ then $\fr{so}(n)$ is simple, and thus any $\op{Ad}$-invariant inner product is homothetic to the negative of the Killing form $B(X,Y)=(n-2)\op{Trace}(XY)$. If $n=4$ then $\fr{so}(n)\equiv\fr{so}(3)\oplus \fr{so}(3)$, and thus any $\op{Ad}$-invariant inner product is homothetic to the negative of the one-parameter family $B_1+\lambda B_2$, $\lambda>0$, where $B_1$ denotes the Killing form of the first simple factor $\fr{so}(3)$ and $B_2$ denotes the Killing form of the second simple factor $\fr{so}(3)$.}
\end{remark}

As a result of Theorem \ref{main1} and Proposition \ref{UnivCover}, we obtain the following.

\begin{corol}\label{corol_o(n)}
Let $G/H$ be one of the spaces $\OO(n)/\OO(n_1)\times \cdots \times \OO(n_s)$ or $\SO(n)/\op{S}(\OO(n_1)\times \cdots \times \OO(n_s))$, where $0<n_1+\cdots +n_s\leq n$, $n_j>1$.  A $G$-invariant Riemannian metric on $G/H$ is geodesic orbit if and only if it is normal.
\end{corol} 

The second main result is the following.

\begin{theorem}\label{main2}Let $G/H$ be the space $\U(n)/\U(n_1)\times \cdots \times \U(n_s)$, where $n_1+\cdots+n_s\leq n$, and let $N_G(H)$ be the normalizer of $H$ in $G$. If $n_1+\dots +n_s=n$, then a $G$-invariant Riemannian metric on $G/H$ is geodesic orbit if and only if it is the normal metric induced from the $\op{Ad}$-invariant inner product $B(X,Y)=-\op{Trace}(XY)$ in $\fr{u}(n)$.  If $n_1+\cdots +n_s<n$, then a $G$-invariant Riemannian metric $g$ on $G/H$ is geodesic orbit if and only if $g=g_{\mu}$, $\mu>0$, where $g_{\mu}$ denotes a one-parameter family of deformations of the normal metric induced from the inner product $B$, along the center of the group $N_G(H)/H$.
\end{theorem}

The case $G/H=\Sp(n)/\Sp(n_1)\times \cdots \times \Sp(n_s)$ has been treated in \cite{Ar-Sou-St}.

\smallskip
We note that the metrics in Theorem \ref{main2} generalize the g.o. metrics on the Berger spheres $\U(n)/\U(n-1)$ (\cite{Ni1}) and the g.o. metrics on the complex Stiefel manifolds $\U(n)/\U(n-k)$ (\cite{So1}).  We also note that the g.o. metrics on the related class of real flag manifolds were recently studied in \cite{Neg}.
 Among other results, it is shown in \cite{Neg} that every g.o. metric on the real flag manifold $\SO(n)/\op{S}(\OO(n_1)\times \cdots \times \OO(n_s))$, $n_1+\cdots +n_s=n$, is normal, which is a special case of Corollary \ref{corol_o(n)}.

The paper is structured as follows: In Sections \ref{Prel} and \ref{Prop}, some preliminary facts for homogeneous spaces and g.o. spaces are given respectively.  Theorems \ref{main1} and \ref{main2} are proved in Sections \ref{proof1} and \ref{proof2} respectively. To this end, we firstly compute the isotropy representation in terms of a suitable basis for each of the spaces (Sections \ref{isotropy1} and \ref{isotropy2} respectively) and then we apply simplification results from \cite{Ni2} and \cite{So1} in order to complete the proofs (Sections \ref{proof1} and \ref{proof2} respectively).

\medskip
\noindent 
{\bf Acknowledgments.}  This research is co-financed by Greece and the European Union (European Social Fund- ESF) through the Operational Programme ``Human Resources Development, Education and Lifelong Learning 2014-2020" in the context of the project ``Geodesic orbit metrics on homogeneous spaces of classical Lie groups" (MIS 5047124). 
Remark \ref{SO(4)} emerged through a discussion of the first author with Professor McKenzie Wang and Remark \ref{referee} was kindly pointed out to us by the referee.  The authors acknowledge both of them for their clarifications.

\section{Invariant metrics on homogeneous spaces}\label{Prel}

Let $G/H$ be a homogeneous space with origin $o=eH$ and assume that $G$ is compact.  Let $\fr{g},\fr{h}$ be the Lie algebras of $G,H$ respectively.  Moreover, let $\op{Ad}:G\rightarrow \op{Aut}(\fr{g})$ and $\op{ad}:\fr{g}\rightarrow \op{End}(\fr{g})$ be the adjoint representations of $G$ and $\fr{g}$ respectively, where $\op{ad}(X)Y=[X,Y]$.  Since $G$ is compact, there exists an $\op{Ad}$-invariant (and hence $\op{ad}$ skew-symmetric) inner product $B$ on $\fr{g}$, which we henceforth fix.  In turn, we have a $B$-orthogonal \emph{reductive decomposition}
\begin{equation}\label{ReducDecompos}
\fr{g}=\fr{h}\oplus \fr{m},
\end{equation}
 where the subspace $\fr{m}$ is $\op{Ad}(H)$-invariant (and $\op{ad}(\fr{h})$-invariant) and is naturally identified with the tangent space of $G/H$ at the origin.

 A Riemannian metric $g$ on $G/H$ is called $G$-invariant if for any $x\in G$, the left translations $\tau_x:G/H\rightarrow G/H$, $pH\mapsto (xp)H$, are isometries of $(G/H,g)$.  The $G$-invariant metrics are in one to one correspondence with $\op{Ad}(H)$-invariant inner products $\langle \ , \ \rangle$ on $\fr{m}$.  Moreover, any such product corresponds to a unique endomorphism $A:\fr{m}\rightarrow \fr{m}$, called the \emph{corresponding metric endomorphism}, that satisfies 
 \begin{equation}\label{MetEnd}\langle X,Y \rangle =B(AX,Y) \ \ \makebox{for all} \ \ X,Y\in \fr{m}.\end{equation} 
 It follows from Equation \eqref{MetEnd} that the metric endomorphism $A$ is symmetric with respect $B$, positive definite and $\op{Ad}(H)$-equivariant, that is $(\op{Ad}(h)\circ A)(X)=(A\circ \op{Ad}(h))(X)$ for all $h\in H$ and $X\in \fr{m}$.  Conversely, any endomorphism on $\fr{m}$ with the above properties determines a unique $G$-invariant metric on $G/H$.  

Since $A$ is diagonalizable, there exists a decomposition $\fr{m}=\bigoplus_{j=1}^l\fr{m}_{\lambda_j}$ into eigenspaces $\fr{m}_{\lambda_j}$ of $A$ corresponding to distinct eigenvalues $\lambda_j$. Each eigenspace $\fr{m}_{\lambda_j}$ is $\op{Ad}(H)$-invariant.  When an $\op{Ad}$-invariant inner product $B$ and a $B$-orthogonal reductive decomposition \eqref{ReducDecompos} have been fixed, we will make no distinction between a $G$-invariant metric $g$ and its corresponding metric endomorphism $A$.

The form of the $G$-invariant metrics on $G/H$ depends on the \emph{isotropy representation} $\op{Ad}^{G/H}:H\rightarrow \op{Gl}(\fr{m})$, defined by $\op{Ad}^{G/H}(h)X:=(d\tau_h)_o(X)$, $h\in H$, $X\in \fr{m}$.  We consider a $B$-orthogonal decomposition
\begin{equation}\label{IsotropyDecom}\fr{m}=\fr{m}_1\oplus \cdots \oplus\fr{m}_s,\end{equation}
 into $\op{Ad}^{G/H}$-invariant and irreducible submodules. We recall that two submodules $\fr{m}_i$ and $\fr{m}_j$ are equivalent if there exists an $\op{Ad}^{G/H}$-equivariant isomorphism $\phi:\fr{m}_i\rightarrow \fr{m}_j$. The simplest case occurs when all the submodules $\fr{m}_i$ are pairwise inequivalent.  Then any $G$-invariant metric $A$ on $G/H$ has a diagonal expression with respect to decomposition \eqref{IsotropyDecom}.  In particular, $\left.A\right|_{\fr{m}_j}=\lambda_j\op{Id}$, $j=1,\dots,s$.

The next proposition is useful for computing the isotropy representation of a reductive homogeneous space. 

\begin{prop}\textnormal{(\cite{Ar1})}\label{isotrepr}
Let $G/H$ be a reductive homogeneous space and let $\fr{g} = \fr{h}\oplus\fr{m}$ be a reductive decomposition of $\fr{g}$.  Let $h\in H$, $X\in \fr{h}$ and $Y\in\fr{m}$.  Then the adjoint representation of $G$ decomposes as
$
\Ad^{G}(h)(X + Y) = \Ad^{G}(h)X + \Ad^{G/H}(h)Y
$
that is, the restriction $\Ad^{G}\big|_{H}$ splits into the sum $\Ad^{H}\oplus\Ad^{G/H}$.  We denote by $\chi$ the representation $\Ad^{G/H}$. 
\end{prop}

The following lemma provides a simple condition for proving that two $\Ad^{G/H}$-submodules are inequivalent.

\begin{lemma}\label{EquivalentLemma}Let $G/H$ be a homogeneous space with  reductive decomposition $\fr{g}=\fr{h}\oplus \fr{m}$ and let $\fr{m}_i,\fr{m}_j\subseteq \fr{m}$ be submodules of the isotropy representation $\Ad^{G/H}$.  Assume that for any pair of non-zero vectors $X\in \fr{m}_i$, $Y\in \fr{m}_j$, there exists a vector $a\in \fr{h}$ such that $[a,X]=0$ and $[a,Y]\neq 0$.  Then the submodules $\fr{m}_i,\fr{m}_j$ are $\Ad^{G/H}$-inequivalent.\end{lemma}

\begin{proof}
If $\fr{m}_i,\fr{m}_j$ were equivalent, then there exists an $\Ad^{G/H}$-equivariant isomorphism $\phi:\fr{m}_i\rightarrow \fr{m}_j$.  Let $X$ be a non-zero vector in $\fr{m}_i$ and set $Y:=\phi(X)\in \fr{m}_j$.  The $\operatorname{Ad}^{G/H}$-equivariance of $\phi$ implies that $\phi$ is $\op{ad}_{\fr{h}}$-equivariant, and hence, $\phi([a,X])=[a,Y]$ for any $a\in \fr{h}$.  However, $\phi$ is an isomorphism, therefore, $[a,X]$ is non-zero if and only if $[a,Y]$ is non-zero, which contradicts the hypothesis of the lemma.  Hence, the submodules $\fr{m}_i,\fr{m}_j$ are inequivalent.\end{proof}

\begin{remark}{\rm  For any two $\Ad^{G/H}$-submodules $\fr{m}_1,\fr{m}_2$, we denote by $[\fr{m}_1,\fr{m}_2]$ the space generated by the vectors $[X_1,X_2]$ where $X_1\in \fr{m}_1$ and $X_2\in \fr{m}_2$.  Similarly, denote by $[\fr{h},\fr{m}_1]$ the space generated by the vectors $[a,X_1]$ where $a\in \fr{h}$ and $X_1\in \fr{m}_1$.  If $\fr{m}_1,\fr{m}_2$ are $B$-orthogonal then $[\fr{m}_1,\fr{m}_2]\subseteq \fr{m}$.  Indeed, $[\fr{m}_1,\fr{m}_2]$ is $B$-orthogonal to $\fr{h}$ because $B([\fr{m}_1,\fr{m}_2],\fr{h})\subseteq B(\fr{m}_1,[\fr{m}_2,\fr{h}])\subseteq B(\fr{m}_1,\fr{m}_2)=\{0\}$.  Moreover, $[\fr{m}_1,\fr{m}_2]$ is also an $\Ad^{G/H}$-submodule of $\fr{m}$ and $[\fr{h},\fr{m}_1]$ is an $\Ad^{G/H}$-submodule of $\fr{m}_1$.}
\end{remark}

\section{Properties of geodesic orbit spaces}\label{Prop}

\begin{definition}
A $G$-invariant metric $g$ on $G/H$ is called a geodesic orbit metric (g.o. metric) if any geodesic of $(G/H,g)$ through $o$ is an orbit of a one parameter subgroup of $G$.  Equivalently, $g$ is a geodesic orbit metric if for any geodesic $\gamma$ of $(G/H,g)$ though $o$ there exists a non-zero vector $X\in \fr{g}$ such that $\gamma(t)=\exp (tX)\cdot o$, $t\in \mathbb R$.  The space $(G/H,g)$ is called a geodesic orbit space (g.o. space). 

\end{definition}

Let $G/H$ be a homogeneous space with $G$ compact.  We fix an $\op{Ad}$-invariant inner product $B$ on $\fr{g}$ and consider the $B$-orthogonal reductive decomposition \eqref{ReducDecompos}.  Moreover, identify each $G$-invariant metric on $G/H$ with the corresponding metric endomorphism $A:\fr{m}\rightarrow \fr{m}$ satisfying Equation \eqref{MetEnd}.  We have the following condition.

\begin{prop}\emph{(\cite{Al-Ar}, \cite{So1})}\label{GOCond} The metric $A$ on $G/H$ is geodesic orbit if and only if for any vector $X\in\fr{m}$ there exists a vector $a\in \fr{h}$ such that  
\begin{equation}\label{cor}[a+X,AX]=0.\end{equation}
\end{prop}
The following result, which we will call the \emph{normalizer lemma}, can be used to simplify the necessary form of the g.o. metrics on $G/H$ by using the normalizer $N_G(H^0)$.

\begin{lemma}\label{NormalizerLemma}\emph{(\cite{Ni2})} 
The inner product $\langle \ ,\ \rangle$ in {\rm (\ref{MetEnd})}, generating the metric of a geodesic orbit Riemannian space $(G/H,g)$, is not only $\op{Ad}(H)$-invariant but also $\op{Ad}(N_G(H^0))$-invariant, where $N_G(H^0)$ is the normalizer of the identity component $H^0$ of the group $H$ in $G$.\end{lemma} 

\noindent As a result of the normalizer lemma, the metric endomorphism $A$ of a g.o. metric on $G/H$ is $\op{Ad}(N_G(H^0))$-equivariant. We will now state a complementary result to the normalizer lemma for compact spaces, that characterizes the restriction of a g.o. metric to the compact Lie group $N_G(H^0)/H^0$. 

\begin{lemma}\label{DualNormalizer}
Let $G$ be a compact Lie group and let $(G/H,g)$ be a geodesic orbit space with the $B$-orthogonal reductive decomposition $\fr{g}=\fr{h}\oplus \fr{m}$, where $B$ is an $\op{Ad}$-invariant inner product on $\fr{g}$. Let $A:\fr{m}\rightarrow \fr{m}$ be the corresponding metric endomorphism of $g$, and let $\fr{n}\subseteq \fr{m}$ be the Lie algebra of the compact Lie group $N_G(H^0)/H^0$.  Then the restriction of $A$ to $\fr{n}$   defines a bi-invariant metric on $N_G(H^0)/H^0$.
\end{lemma}

\begin{proof}  We denote by $\fr{n}_{\fr{g}}(\fr{h})\subset \fr{g}$ the Lie algebra of $N_G(H^0)$.  We have a $B$-orthogonal decomposition $\fr{n}_{\fr{g}}(\fr{h})=\fr{h}\oplus \fr{n}$, where $\fr{n}$ coincides with the Lie algebra of $N_G(H^0)/H^0$.  Moreover, we have a $B$-orthogonal decomposition $\fr{m}=\fr{n}\oplus \fr{p}$, where $\fr{p}$ coincides with the tangent space of $G/N_G(H^0)$ at the origin.  By the normalizer lemma, the restriction of $A$ on $\fr{p}$ defines an invariant metric on $G/N_G(H^0)$, and thus $A\fr{p}\subseteq \fr{p}$.  By taking into account the symmetry of $A$ with respect to the product $B$, we deduce that $B(A\fr{n},\fr{p})=B(\fr{n},A\fr{p})\subseteq B(\fr{n},\fr{p})=\{0\}$.  Hence, the image $A\fr{n}$ is $B$-orthogonal to $\fr{p}$ which, along with decomposition $\fr{m}=\fr{n}\oplus \fr{p}$, yields $A\fr{n}\subseteq \fr{n}$. Therefore, the restriction $\left.A\right|_{\fr{n}}:\fr{n}\rightarrow \fr{n}$ defines a left-invariant metric on $N_G(H^0)/H^0$.  Since $A$ is a g.o. metric on $G/H$, Proposition \ref{GOCond} implies that for any $X\in \fr{n}$ there exists a vector $a\in \fr{h}$ such that $0=[a+X,AX]=[a+X,\left.A\right|_{\fr{n}}X]$. Therefore, by the same proposition, $\left.A\right|_{\fr{n}}$ defines a g.o. metric on $N_G(H^0)/H^0$.  On the other hand, any left-invariant g.o. metric on a Lie group is necessarily bi-invariant (\cite{Al-Ni}), and hence $\left.A\right|_{\fr{n}}$ is a bi-invariant metric on $N_G(H^0)/H^0$.\end{proof}

\begin{remark}\label{referee}
{\rm As an alternative to the above proof, it was observed by the referee that Lemma \ref{DualNormalizer} follows easily from Lemma \ref{NormalizerLemma}, if we take into account that any $\op{Ad}(N_G(H^0))$-invariant Riemannian metric on the Lie group $N_G(H^0)/H^0$ is
generated by a suitable bi-invariant Riemannian metric on $N_G(H^0)/H^0$.}
\end{remark}

 \begin{remark}\label{conclusion} {\rm Let $\fr{p}\subset \fr{m}$ be the tangent space of $G/N_G(H^0)$, let $\fr{n}\subset \fr{m}$ be the Lie algebra of $N_G(H^0)/H^0$, and consider the decomposition $\fr{m}=\fr{n}\oplus \fr{p}$.  By combining Lemmas \ref{NormalizerLemma} and \ref{DualNormalizer}, we conclude that the metric endomorphism $A$ corresponding to a g.o. metric on $G/H$ has the block-diagonal form
$$A=\begin{pmatrix}\left.A\right|_{\fr{n}} & 0\\
0& \left.A\right|_{\fr{p}}\end{pmatrix}.
$$}
\end{remark}

The next lemma describes the invariant g.o. metrics on compact Lie groups, in terms of their metric endomorphism with respect to an $\op{Ad}$-invariant inner product $B$.  Recall that the Lie algebra $\fr{g}$ of a compact Lie group $G$ has the (Lie algebra) direct sum decomposition $\fr{g}=\fr{g}_1\oplus \cdots \oplus \fr{g}_k\oplus \fr{z}$, where $\fr{g}_j$ are the simple ideals of $\fr{g}$ and $\fr{z}$ is its center.

\begin{lemma}\label{GOLieGroups}\emph{(\cite{Al-Ni}, \cite{So1})} 
Let $G$ be a compact Lie group with Lie algebra $\fr{g}=\fr{g}_1\oplus \cdots \oplus \fr{g}_k\oplus \fr{z}$.  A left invariant metric $A$ on $G$ is a g.o. metric if and only if it is bi-invariant.  In particular, $A$ is a g.o. metric if and only if 
\begin{equation*}A=\begin{pmatrix} 
 \lambda_1\left.\op{Id}\right|_{\fr{g}_1} & 0 & \cdots &0\\
  \vdots & \ddots & \cdots &\vdots\\
  0&\cdots &\lambda_k\left.\op{Id}\right|_{\fr{g}_k} &0\\
  0& \cdots & 0 & \left.A\right|_{\fr{z}}
  \end{pmatrix}, \ \lambda_j>0.
\end{equation*}
\end{lemma}

 We  set some notation. For a subspace $W$ of a vector space $V$ we write $V=W\oplus W^{\bot}$ with respect to some inner product $B$ on $V$.  Then, for $v\in V$ it is $v=w+w'$, where $w\in W$ and $w'\in W^{\bot}$.  We say that $v$ 
  has \emph{non zero projection on $W$} if $w\neq 0$. 
  Moreover, we  say that a subset $S$ of $V$ has non zero projection on $W$ if there exists a vector $v\in S$ that has non zero projection on $W$.  
  
  The following lemma is useful, since it will enable us to equate some of the eigenvalues of a g.o. metric.

\begin{lemma}\label{EigenEq}\emph{(\cite{So1})}
 Let $(G/H,g)$ be a g.o. space with $G$ compact and with corresponding metric endomorphism $A$ with respect to an $\op{Ad}$-invariant inner product $B$. Let $\fr{m}$ be the $B$-orthogonal complement of $\fr{h}$ in $\fr{g}$.\\ 
\textbf{1.} Assume that $\fr{m}_1,\fr{m}_2$ are $\operatorname{ad}(\fr{h})$-invariant, pairwise $B$-orthogonal subspaces of $\fr{m}$ such that $[\fr{m}_1,\fr{m}_2]$ has non-zero projection on ${(\fr{m}_1\oplus \fr{m}_2)^\bot}$. Let $\lambda_1,\lambda_2$ be eigenvalues of $A$ such that $\left.A\right|_{\fr{m}_i}=\lambda_i\op{Id}$, $i=1,2$.  Then $\lambda_1=\lambda_2$.\\
\textbf{2.} Assume that $\fr{m}_1,\fr{m}_2,\fr{m}_3$ are $\operatorname{ad}(\fr{h})$-invariant, pairwise $B$-orthogonal subspaces of $\fr{m}$ such that $[\fr{m}_1,\fr{m}_2]$ has non-zero projection on $\fr{m}_3$. Let $\lambda_1,\lambda_2,\lambda_3$ be eigenvalues of $A$ such that $\left.A\right|_{\fr{m}_i}=\lambda_i\op{Id}$, $i=1,2,3$.  Then $\lambda_1=\lambda_2=\lambda_3$.
\end{lemma}
Finally, the following result will be used  in relating  
g.o. metrics on a space $G/H$ to  g.o. metrics on its universal covering space $\widetilde{G}/\widetilde{H}$ (see also \cite{So3}).
Recall that a $G$-invariant metric is called {\it standard} if it is induced by the negative of the Killing form on $\fr{g}$.

\begin{prop}\label{UnivCover}
Let $G/H$, $\widetilde{G}/\widetilde{H}$ be homogeneous spaces with $G$ compact and $\widetilde{H}$ connected, such that the Lie algebras of $G$ and $\widetilde{G}$ coincide and the Lie algebras of $H$ and $\widetilde{H}$ coincide. If every ($\widetilde{G}$-invariant) g.o. metric on $\widetilde{G}/\widetilde{H}$ is normal (resp. standard), then every ($G$-invariant) g.o. metric on $G/H$ is also normal (resp. standard).  Moreover, the converse is true if $H$ is connected.
\end{prop} 
\begin{proof} Let $\fr{g}$ denote the Lie algebra of the groups $G$ and $\widetilde{G}$, and let $\fr{h}$ denote the Lie algebra of the groups $H$ and $\widetilde{H}$.  Let $B$ be an $\op{Ad}$-invariant inner product on $\fr{g}$ and consider the $B$-orthogonal decomposition $\fr{g}=\fr{h}\oplus \fr{m}_B$. Then $\fr{m}_B$ can be identified with the tangent spaces $T_o(G/H)$ and $T_o(\widetilde{G}/\widetilde{H})$. Let $g$ be a $G$-invariant g.o. metric on $G/H$.  We will prove that $g$ is normal.  The proof will be completed in three steps. 

\textbf{Step 1.} \emph{The $G$-invariant metric $g$ on $G/H$ induces a $\widetilde{G}$-invariant metric $\widetilde{g}$ on $\widetilde{G}/\widetilde{H}$}. Indeed, let $A_B:\fr{m}_{B}\rightarrow \fr{m}_{B}$ be the corresponding metric endomorphism of $g$. The endomorphism $A_B$ is $\op{Ad}_H$-invariant and hence $\op{ad}_{\fr{h}}$-equivariant.  Given that $\widetilde{H}$ is connected, the $\op{ad}_{\fr{h}}$-equivariance of $A_B$ yields its $\op{Ad}_{\widetilde{H}}$-equivariance.  Therefore, $A_B$ defines a $\widetilde{G}$-invariant metric $\widetilde{g}$ on $\widetilde{G}/\widetilde{H}$.     

\textbf{Step 2.} \emph{The metric $\widetilde{g}$ is also a g.o. metric on $\widetilde{G}/\widetilde{H}$.}  Indeed, since $g$ is a go. metric then the corresponding metric endomorphism $A_B$ satisfies Proposition \ref{GOCond}. Since $A_B$ is also the metric endomorphism of $\widetilde{g}$, Proposition \ref{GOCond} implies that $\widetilde{g}$ is a g.o. metric on $\widetilde{G}/\widetilde{H}$. 

\textbf{Step 3.} \emph{The metric $g$ is normal.} Indeed, since $\widetilde{g}$ is a g.o. metric on $\widetilde{G}/\widetilde{H}$, by hypothesis it is also normal. Hence, there exists an $\op{Ad}$-invariant inner product $B^{\prime}$ on $\fr{g}$ and a $B^{\prime}$-orthogonal decomposition $\fr{g}=\fr{h}\oplus \fr{m}_{B^{\prime}}$ such that the corresponding metric endomorphism $A_{B^{\prime}}:\fr{m}_{B^{\prime}}\rightarrow \fr{m}_{B^{\prime}}$ of $\widetilde{g}$ satisfies $A_{B^{\prime}}=\lambda\op{Id}$, $\lambda>0$.  Since $A_{B^{\prime}}$ coincides with the metric endomorphism of $g$, we conclude that $g$ is also normal.  
\end{proof} 
\begin{corol}\label{UnivCover1}If every ($\widetilde{G}$-invariant) g.o. metric on the universal cover $\widetilde{G}/\widetilde{H}$ of $G/H$ is normal (resp. standard), then every ($G$-invariant) g.o. metric on $G/H$ is also normal (resp. standard).
\end{corol}

\section{The space $M=G/H=\SO(n)/\SO(n_1)\times\cdots\times\SO(n_s)$, $n_1+\cdots +n_s\leq n$, $n_j>1$.}\label{O(n)}

\subsection{Isotropy representation of $G/H=\SO (n)/\SO(n_1)\times\cdots\times\SO(n_s)$}\label{isotropy1}

We set 
\begin{equation}\label{n_0}n_0:= n-(n_1+\cdots + n_s).\end{equation}
 We view $H = \SO(n_{1})\times\cdots\times\SO(n_{s})$ embedded diagonally in $\SO(n)$, so  that
$H\cong\begin{pmatrix}
\Id _{n_0} & 0 \\
 0 & H
 \end{pmatrix}. 
 $
 Hence
$$
\fr{h}=\begin{pmatrix}
0_{n_0} &  & & 0\\
   & \fr{so}(n_1) &  & \\
   &  & \ddots & \\
0   &  &  & \fr{so}(n_s)\\
\end{pmatrix},
$$
where $0_{n_0}$ is the $n_0\times n_0$ zero matrix.  We remark that the above embedding of $H$ is equivalent (via conjugation in $\SO(n)$) to any block-diagonal embedding of the factors $\SO(n_j)$.

Recall that if $\pi : G \to \Aut(V),$ \ $\pi' : G \to \Aut(W)$ are two representations of $G$, then for the second exterior power the following identity is valid: 
$
\wedge^{2}(\pi\oplus\pi') = \wedge^{2}\pi\oplus\wedge^{2}\pi'\oplus(\pi\otimes\pi').
$

 Denote by $\lambda_{n}: \SO(n)\to \Aut(\bb{R}^{n})$  the standard representation of $\SO(n)$.
 Then the adjoint representation $\Ad^{\SO(n)}$ of $\SO(n)$ (and $\Ad^{\OO(n)}$ of $\OO(n)$)  is equivalent to $\wedge^{2}\lambda_{n}$.
  
Let $\sigma _{n_i}: \SO(n_{1})\times\cdots\times\SO(n_{s})\to\SO(n_{i})$ be the projection onto the $i$-factor and $p_{i} = \lambda_{n_{i}}\circ\sigma_{n_{i}}$
 be the standard representation of $H$, i.e.  
$$ 
\SO(n_{1})\times\cdots\times\SO(n_{s})\stackrel{\sigma_{n_{i}}}{\longrightarrow}\SO(n_{i})\stackrel{\lambda_{n_{i}}}{\longrightarrow}\Aut(\bb{R}^{n_{i}}). 
$$
 Then we have:
\begin{eqnarray*}
\Ad^{G}\big|_{H} &=& \wedge^2\lambda _{n}\big|_{H} = \wedge^2(p_{1}\oplus\cdots\oplus p_{s} \oplus \mathbbm{1}_{n_0}) = \wedge^{2}p_{{1}}\oplus\wedge^{2}p_{{2}}\oplus\cdots\oplus\wedge^{2}p_{s}\oplus\wedge^{2}\mathbbm{1}_{n_0}\\
&&\oplus\{(p_{{1}}\otimes p_{{2}})\oplus\cdots\oplus(p_{{1}}\otimes p_{{s}})\}\oplus\{(p_{{2}}\otimes p_{{3}})\oplus\cdots\oplus (p_2\otimes p_{s})\} \oplus \cdots\oplus (p_{s-1}\otimes p_s) \\
&&
\oplus(p_{{1}}\otimes \mathbbm{1}_{n_0})\oplus(p_{{2}}\otimes \mathbbm{1}_{n_0})\oplus\cdots\oplus(p_{{s}}\otimes \mathbbm{1}_{n_0}),
\\
\end{eqnarray*}
where $\wedge^{2}\mathbbm{1}_{n_0}$ is the sum of $\binom{n_0}{2}$  trivial representations. 

 The dimension of the representation  $\wedge^{2}p_{{1}}\oplus\wedge^{2}p_{{2}}\oplus\cdots\oplus\wedge^{2}p_{s}$ is $\binom{n_1}{2} + \binom{n_2}{2} + \cdots + \binom{n_s}{2}$ and is equal to the dimension of the adjoint representation of  $H = \SO(n_{1})\times\cdots\times\SO(n_{s})$, that is 
$\Ad ^H= \wedge^{2}p_{{1}}\oplus\wedge^{2}p_{{2}}\oplus\cdots\oplus\wedge^{2}p_{s}$. 
Therefore, by Proposition \ref{isotrepr}
the isotropy representation of $G/H$ is given by
\begin{eqnarray}\label{isotropyOO}
\chi &=& \wedge^{2}\mathbbm{1}_{n_0}\oplus (p_{{1}}\otimes p_{{2}})\oplus\cdots\oplus(p_{{1}}\otimes p_{{s}})\oplus (p_{{2}}\otimes p_{{3}})\oplus\cdots\oplus (p_2\otimes p_{s}) \oplus\cdots\oplus (p_{s-1}\otimes p_s)
 \nonumber\\
&& \oplus(p_{{1}}\otimes \mathbbm{1}_{n_0})\oplus(p_{{2}}\otimes \mathbbm{1}_{n_0})\oplus\cdots\oplus(p_{{s}}\otimes \mathbbm{1}_{n_0}).
\end{eqnarray}
 The dimension of each of $p_i\otimes p_j$ is $n_i\times n_j$ and each of $p_i\otimes \mathbbm{1}_{n_0}$, $i = 1,2,\ldots, s$, contains $n_0$ equivalent representations of dimension $n_i$.   
 
 \begin{remark}\label{SO(4)}
 \textnormal{ Each of the representations $p_i\otimes p_j$ corresponds to the isotropy representation of the Grassmannian $\SO(n_i+n_j)/\SO(n_i)\times \SO(n_j)$ and is irreducible unless $n_i=n_j=2$.  If $n_i=n_j=2$, then the summand $p_i\otimes p_j$ in (\ref{isotropyOO}) reduces into two 2-dimensional irreducible non-equivalent representations.   This summand corresponds to the isotropy representation of the Grassmannian $\SO(4)/\SO(2)\times\SO(2)$.  However, the reducibility does not carry to the Grassmannian $\OO(4)/\OO(2)\times\OO(2)$, which is isotropy irreducible.    Therefore, even if the isotropy subgroup $\OO(n_1)\times\cdots\times\OO(n_s)$ contains at least two $\OO(2)$-factors, the isotropy representation of $\OO(n)/\OO(n_1)\times\cdots \times\OO(n_s)$ is still (\ref{isotropyOO}).} 
 \end{remark}

Expression (\ref{isotropyOO}) induces a decomposition of the tangent space $\fr{m}$ of $G/H$  as 
\begin{equation}\label{decompOO}
\fr{m} = \fr{n}_1\oplus\cdots\oplus\fr{n}_{\binom{n_0}{2}} \bigoplus_{1\le i<j\le s}\fr{m}_{ij}\bigoplus_{j=1}^{s}\fr{m}_{0j},
\end{equation}
where $\dim(\fr{n}_i) = 1$, $\fr{m}_{0j} = \fr{m}_1^{j}\oplus\fr{m}_2^{j}\oplus\cdots\oplus\fr{m}_{n_0}^{j}$ with $\fr{m}_{\al}^{j}\cong\fr{m}_{\beta}^{j}$, $\al \neq \beta$ and $\dim(\fr{m}^{j}_{\ell})= n_j$, $\ell=1,2,\ldots,n_0$.
Note that $\fr{n}_1\oplus\cdots\oplus\fr{n}_{\binom{n_0}{2}}\cong\fr{so}(n_0)$. Moreover, if $n_0=0$ or $n_0=1$ then there are no trivial submodules and hence $\fr{so}(n_0)=\{0\}$.

We now give explicit matrix representations of the modules $\fr{m}_{ij}$ and $\fr{m}_{0j}$. We consider the $\Ad(\SO(n))$-invariant inner product $B:\fr{so}(n)\times\fr{so}(n)\to\mathbb{R}$ given by
\begin{equation}\label{KillSO(n)}
B(X, Y)=-{\rm Trace}(XY), \quad X, Y\in\fr{so}(n),
\end{equation}
and obtain 
a $B$-orthogonal decomposition 
$\fr{g}=\fr{h}\oplus \fr{m}$,
 where $\fr{h}=\fr{so}(n_1)\oplus \cdots \oplus \fr{so}(n_s)$ and $\fr{m}\cong T_o(G/H)$.  
 We consider a basis of $\fr{g}=\fr{so}(n)$ as follows:\\
Let $M_{n}\mathbb R$ be the set of real $n\times n$ matrices and let $E_{ab}\in M_{n}\mathbb R$ be the matrix with 1 in the $(a,b)$-entry and zero elsewhere.  For $1\le a<b\le n$ we set

\begin{equation}\label{oe}e_{ab}:=E_{ab}-E_{ba}.
\end{equation} 
 Note that $e_{ab}=-e_{ba}$.
The set 

\begin{equation}\label{obasis}\mathcal{B}:=\left\{ {e_{ab}: 1\le a<b\le n} \right\},\end{equation}  

\noindent constitutes a basis of $\fr{so}(n)$, which is orthogonal with respect to $B$.
The proof of the following lemma is immediate.

\begin{lemma}\label{rel}
The only non zero bracket relations among the vectors {\rm (\ref{oe})} are
$[e_{ab}, e_{bc}]=e_{ac}$, for  $a,b,c$ distinct.
\end{lemma}

A choice for the modules in the decomposition (\ref{decompOO}) is the following:
\begin{eqnarray*}
\fr{m}_{ij}&=&\Span\{e_{ab}:  n_0+n_1+\cdots +n_{i-1}+1\le a\le 
n_0+n_1+\cdots +n_i, \\
 && \ \ \ \ \ \ \ \ \ \ \ \quad  n_0+n_1+\cdots +n_{j-1}+1\le b\le n_0+n_1+\cdots +n_j\}, \ \ 1\leq i<j\leq s.\\
\fr{m}_{0j}&=&\Span\{e_{ab}: 1\le a\le n_0,\  n_0+n_1+\cdots +n_{j-1}+1\le b\le 
n_0+n_1+\cdots +n_j\}, \ \ 1\leq j\leq s. \\
\fr{so}(n_0) &=& \Span\{e_{ab}: 1\le a<b\le n_0\}.
\end{eqnarray*}
The equivalent modules in the decomposition of $\fr{m}_{0j}$ are given by
$$
\fr{m}^j_\ell =\Span\{e_{\ell b}: n_0+n_1+\cdots +n_{j-1}+1\le b\le 
n_0+n_1+\cdots +n_j \}, \ \ell = 1, \dots , n_0.
$$
Also,
$$
\fr{so}(n_j)=\Span\{e_{ab}: n_0+n_1+\cdots +n_{j-1}+1\le a<b\le n_0+n_1+\cdots +n_j\}, \quad j =1, \dots , s.
$$
Hence, we obtain the $B$-orthogonal decomposition
\begin{equation}\label{dd}\fr{m}=\fr{n}\oplus \fr{p},
\end{equation}
where
$$\fr{n}=\fr{so}(n_0), \quad \fr{p}=\bigoplus_{1\le i<j\le s}\fr{m}_{ij}\bigoplus_{j=1}^{s}\fr{m}_{0j}.
$$
The above decomposition can be depicted in the following matrix, which shows the upper triangular part of $\fr{so}(n)$:

$$
\begin{pmatrix}
\fr{so}(n_0) & \fr{m}_{01} & \fr{m}_{02} & \fr{m}_{03} & \cdots & \fr{m}_{0s}\\
        &  0_{n_1} & \fr{m}_{12} & \fr{m}_{13} & \cdots & \fr{m}_{1s}\\
    &   & 0_{n_2}  & \fr{m}_{23} & \cdots & \fr{m}_{2s}\\
   & \ast  &      & \ddots & \vdots & \vdots\\
   &   &      &  &  & 0_{n_s}     
\end{pmatrix}
$$
The matrices $\fr{m}_{0j}$ are of size $n_0\times n_j$, the matrices $\fr{m}_{ij}$ are of size $n_i\times n_j$, and the matrices $\fr{so}(n_i)$ of size $n_i\times n_i$. We remark that if $n_0=0$ or $n_0=1$, then $\fr{n}=\{0\}$.  In the former case, the submodules $\fr{m}_{0j}$ are zero while in the latter case they are non zero and irreducible.  

Moreover, using Lemma \ref{rel}, we observe that 
\begin{equation}\label{temp1}  
[\fr{so}(n_i),\fr{m}_{lm}]=\left\{ 
\begin{array}{lll}
\fr{m}_{lm},  \quad \makebox{if $i=l$ or $i=m$} \ \\

\{0\} ,   \quad \makebox{otherwise}
\end{array}
\right., \ \ \ 0\leq i\leq s, \ \ 0\leq l<m \leq s, 
\end{equation}
and 
\begin{equation}\label{SubmoduleBrackets}
[\fr{m}_{ij},\fr{m}_{jl}]=\fr{m}_{il} \ \ \makebox{for all} \ \ 0\leq i<j<l\leq s.
\end{equation}
\begin{remark}\label{SO(4)2}\textnormal{In view of Remark \ref{SO(4)}, if the isotropy subgroup $H=\SO(n_1)\times\cdots\times\SO(n_s)$ contains at least two $\SO(2)$-factors, say $n_i=n_j=2$, then the module $\fr{m}_{ij}$ splits into two $\Ad(H)$-irreducible non equivalent summands  each of dimension 2.
More precisely,  if we set  $I=n_0+n_1+\cdots +n_{i-1}+1$, $J=n_0+n_1+\cdots +n_{j-1}+1$, $K=n_0+n_1+\cdots +n_{i-1}+2$, $L=n_0+n_1+\cdots +n_{j-1}+2$, then
$\fr{m}_{ij}=V_{ij}^1\oplus V_{ij}^2$, where
$V_{ij}^1=\Span\{e_{IJ}-e_{KL}, e_{IL}+e_{KJ}\}$, $V_{ij}^2=\Span\{e_{IJ}+e_{KL}, e_{IL}-e_{KJ}\}$.
}

\smallskip
\textnormal{
For example, for $M=G/H=\SO(14)/\SO(3)\times\SO(2)\times\SO(3)\times\SO(2)\times\SO(2)$ we have $n_0=2, n_1=3, n_2=2, n_3=3, n_4=2, n_5=2$, and
$$
\fr{m}=\fr{n}\oplus\fr{p}=\fr{so}(2)\bigoplus_{1\le i<j\le 5}\fr{m}_{ij}\bigoplus_{j=1}^5\fr{m}_{0j}.
$$
The submodules $\fr{m}_{24}$, $\fr{m}_{25}$ and $\fr{m}_{45}$ split as above.  For instance, $\fr{m}_{24}=V_{24}^1\oplus V_{24}^2$, where
$V_{24}^1=\Span\{e_{6,11}-e_{7, 12}, e_{6,12}+e_{7,11}\}$, $V_{24}^2=\{e_{6,11}+e_{7, 12}, e_{6,12}-e_{7,11}\}$.
}
\end{remark}

\subsection{Proof of Theorem \ref{main1}}\label{proof1}

Any normal metric on $G/H$ induced from an $\Ad$-invariant inner product $B$ on $\fr{g}$ is a g.o. metric, and hence the sufficiency part of the theorem holds trivially.  For the necessity part, assume initially that $n_1=\cdots =n_s=2$ so that $H$ is abelian.  Since $G$ is semisimple, the main results in \cite{So2} imply that any g.o. metric on $G/H$ is normal, and hence Theorem \ref{main1} follows in this case. 

 Now assume that $n_j\neq 2$ for some $j=1,\dots,s$.  Let $g$ be a $G$-invariant g.o. metric on $G/H$.  Assume that $A:\fr{m}\rightarrow \fr{m}$ is the corresponding metric endomorphism satisfying Equation \eqref{MetEnd}.  Recall the spaces $\fr{n}=\fr{so}(n_0)$ and $\fr{p}=\bigoplus_{0\leq i<j\leq s}\fr{m}_{ij}$ defined in Section \ref{isotropy1}, and the decomposition $\fr{m}=\fr{n}\oplus \fr{p}$.  Since $H$ is connected, we have $H^0=H$.  The Lie algebra of the normalizer $N_G(H^0)=N_G(H)$ coincides with $\fr{n}_{\fr{g}}(\fr{h})=\{Y\in \fr{g}: [Y,\fr{h}]\subseteq \fr{h}\}$.  In our case, $\fr{h}=\fr{so}(n_1)\oplus \cdots \oplus \fr{so}(n_s)$ and it is not hard to show using the results in Section \ref{isotropy1} that
\begin{equation*}\fr{n}_{\fr{g}}(\fr{h})=\fr{so}(n_0)\oplus \fr{so}(n_1)\oplus \cdots \oplus \fr{so}(n_s)=\fr{n}\oplus \fr{h}.
\end{equation*}
 Therefore, the tangent space of $G/N_G(H)$ coincides with $\fr{p}$.  The normalizer Lemma \ref{NormalizerLemma} then implies that $\left.A\right|_{\fr{p}}$ defines a $G$-invariant g.o. metric on $G/N_G(H)$. 

 By taking into account Lemma \ref{rel} and the expressions of the subspaces $\fr{m}_{ij}, \fr{so}(n_j), \fr{so}(n_0)$ in terms of the basis $\mathcal{B}$, we deduce that the submodules $\fr{m}_{ij}$, $0\leq i<j\leq s$ are $\op{ad}(\fr{n}_{\fr{g}}(\fr{h}))$-invariant. If $n_i\neq 2$ or $n_j\neq 2$ then the submodules $\fr{m}_{ij}$ are $\op{ad}(\fr{n}_{\fr{g}}(\fr{h}))$-irreducible. On the other hand, in view of Remark \ref{SO(4)2}, $\fr{m}_{ij}=V_{ij}^1\oplus V_{ij}^2$ if $n_i=n_j=2$.  
 
We now claim that the $\op{ad}(\fr{n}_{\fr{g}}(\fr{h}))$-submodules $\fr{m}_{ij}$, $0\leq i<j\leq s$, are pairwise inequivalent.  Indeed, if $\fr{m}_{ij}$ and $\fr{m}_{lm}$, with $0\leq i<j\leq s$, $0\leq l<m\leq s$, are two distinct $\op{ad}(\fr{n}_{\fr{g}}(\fr{h}))$-submodules, then there exists  an index $i_0$ such that one of the following happens:

\smallskip
\noindent
 \centerline{{\bf Case 1.} $i_0=i$ or $i_0=j$ and $i_0\neq l,m$, \quad
{\bf Case 2.} $i_0=l$ or $i_0=m$ and $i_0\neq i,j$. }

\smallskip
\noindent For {\bf Case 1}, we take into account relation \eqref{temp1} and obtain that $[\fr{so}(n_{i_0}),\fr{m}_{ij}]=\fr{m}_{ij}$ and $[\fr{so}(n_{i_0}),\fr{m}_{lm}]=\{0\}$.
 For {\bf Case 2} we have that $[\fr{so}(n_{i_0}),\fr{m}_{ij}]=\{0\}$ and $[\fr{so}(n_{i_0}),\fr{m}_{lm}]=\fr{m}_{lm}$.  By virtue of Lemma \ref{EquivalentLemma}, we deduce that the submodules $\fr{m}_{ij}$ and $\fr{m}_{lm}$ are inequivalent in both cases, which proves the claim.

 Therefore, the restriction of the g.o. metric $\left.A\right|_{\fr{p}}$ of $G/N_G(H)$ to each of the $\op{ad}(\fr{n}_{\fr{g}}(\fr{h}))$-submodules $\fr{m}_{ij}$, $0\leq i<j<l\leq s$, has the diagonal form
\begin{equation}\label{Condition1}
\left.A\right|_{\fr{m}_{ij}}=\left\{ 
\begin{array}{lll}\lambda_{ij}\op{Id},  \qquad\qquad\qquad\qquad\ \ 
 \makebox{if $n_i\neq 2$ or $n_j\neq 2$}\\ \\
\begin{pmatrix} \lambda_{ij}^1\left.\op{Id}\right|_{V_{ij}^1} & 0\\
0& \lambda_{ij}^2\left.\op{Id}\right|_{V_{ij}^2}\end{pmatrix}, \ \ \makebox{if $n_i=n_j=2$}.\end{array}
\right.   
\end{equation}  
 In view of the decomposition $\fr{m}=\fr{n}\oplus \fr{p}$ and Remark \ref{conclusion}, Theorem \ref{main1} now follows from the following two propositions, whose proofs we present in the next subsection.

\begin{prop}\label{Simplif1}
$\left.A\right|_{\fr{p}}=\lambda\op{Id}$, where $\lambda>0$. 
\end{prop}

\begin{prop}\label{Simplif2}
$\left.A\right|_{\fr{n}}=\lambda\op{Id}$, where $\lambda$ is given by {\rm Proposition \ref{Simplif1}}.\qed
\end{prop}

\subsection{Proof of Propositions \ref{Simplif1} and \ref{Simplif2}}

\noindent To prove Proposition \ref{Simplif1}  we need the following lemmas.

\begin{lemma}\label{restr}
Let $A$ be the metric endomorphism of a g.o. metric on $M=\SO(n)/\SO(n_1)\times \cdots \times \SO(n_s)$.  Then for any $r=1,\dots,s$, the restriction of $A$ to the tangent space of $\widetilde{M}:=\SO(n_r+\cdots +n_s)/\SO(n_r)\times \cdots \times \SO(n_s)$ defines a $\SO(n_r+\cdots +n_s)$-invariant g.o. metric on $\widetilde{M}$.
\end{lemma}
\begin{proof} Let $\widetilde{\fr{g}}:=\fr{so}(n_r+\cdots +n_s)$ and $\widetilde{\fr{h}}:=\fr{so}(n_r)\oplus \cdots \oplus \fr{so}(n_s)$ be the Lie algebras of $\SO(n_r+\cdots +n_s)$ and $\SO(n_r)\times \cdots \times \SO(n_s)$ respectively.  Taking into account the notation and results of section \ref{isotropy1}, the tangent space of $\widetilde{M}$ coincides with the $B$-orthogonal complement $\widetilde{\fr{m}}:=\bigoplus_{r\leq i<j\leq s}\fr{m}_{ij}$ of $\widetilde{\fr{h}}$ in $\widetilde{\fr{g}}$.  Relation \eqref{Condition1} implies that $\left.A\right|_{\widetilde{\fr{m}}}$ defines an endomorphism of $\widetilde{\fr{m}}$, which is $\op{ad}(\fr{h})$-equivariant and hence $\op{ad}(\widetilde{\fr{h}})$-equivariant.  Therefore, $\left.A\right|_{\widetilde{\fr{m}}}$ defines an $\SO(n_r+\cdots +n_s)$-invariant metric on $\widetilde{M}$.  Since $A$ defines a g.o. metric on $G/H$, Proposition \ref{GOCond} implies that for any $X\in \widetilde{\fr{m}}$, there exists $a\in \fr{h}$ such that 
\begin{equation}\label{GOCond1} 0=[a+X,AX]=[a+X,\left.A\right|_{\widetilde{\fr{m}}}X].\end{equation}
 On the other hand, relation \eqref{temp1} along with the definition of $\widetilde{\fr{m}}$ imply that $[\fr{h},\widetilde{\fr{m}}]=[\widetilde{\fr{h}},\widetilde{\fr{m}}]$.  Therefore, we may assume that the vector $a$ in expression \eqref{GOCond1} lies in $\widetilde{\fr{h}}$.  By Proposition \ref{GOCond} we conclude that $\left.A\right|_{\widetilde{\fr{m}}}$ defines a g.o. metric on $\widetilde{M}$. \end{proof}

\begin{lemma}\label{CombinatorialLemma1}
Let $R_s=\{\lambda_{ij}: \ 0\leq i<j\leq s\}$ be a set such that $\lambda_{ij}=\lambda_{jk}=\lambda_{ik}$ for all $0\leq i<j<k\leq s$. Then $R_s$ is a  singleton.
\end{lemma}
 \begin{proof}  We will proceed by induction on $s\geq 1$.  If $s=1$ then $R_1=\{\lambda_{01}\}$ and the result holds trivially.  Assume that the lemma holds for $s=N$ and let $s=N+1$.  The assumption that $\lambda_{ij}=\lambda_{jk}=\lambda_{ik}$ for all $0\leq i<j<k\leq s=N+1$ implies that $\lambda_{ij}=\lambda_{jk}=\lambda_{ik}$ for all $0\leq i<j<k\leq N$. By the induction hypothesis, the set $R_N=\{\lambda_{ij}: \ 0\leq i<j\leq N\}$ is a singleton, i.e. $R_N=\{\lambda\}$.  

To prove that $R_{N+1}$ is a singleton, it remains to show that
\begin{equation}\label{remains1}\lambda_{0,N+1}=\lambda_{1,N+1}=\cdots =\lambda_{N,N+1}=\lambda.\end{equation}
 To this end, consider two arbitrary elements $\lambda_{i,N+1}$, $\lambda_{j,N+1}$.  Without loss of generality, assume that $i<j$.  Since $0\leq i<j<N$, we have $\lambda_{ij}\in R_N$ and hence $\lambda_{ij}=\lambda$.  On the other hand, the assumption of the lemma yields $\lambda=\lambda_{ij}=\lambda_{j,N+1}=\lambda_{i,N+1}$.  Since the choice of $i,j$ is arbitrary, Equation \eqref{remains1} is true and thus $R_{N+1}=\{\lambda\}$, which concludes the induction.\end{proof}

\noindent \emph{Proof of Proposition \ref{Simplif1}}. The proof consists of two steps.

\noindent
\textbf{Step 1.} Prove that both quantities $\lambda_{ij}^1$, $\lambda_{ij}^2$ in relation \eqref{Condition1} are equal for all $i,j$, which will imply that 
\begin{equation}\label{psad}\left.A\right|_{\fr{m}_{ij}}=\lambda_{ij}\op{Id}.\end{equation}
\noindent
\textbf{Step 2.}  Prove that all $\lambda_{ij}$ in relation \eqref{psad} are equal.
 
\smallskip
 For \textbf{Step 1}, recall the assumption that at least one of the $n_j$, $j=1,\dots ,s$ is not equal to $2$.  Recall also from definition \eqref{n_0} that $n_0=n-(n_1+\cdots +n_s)$. Without any loss of generality, we may apply a conjugation $\phi\in \op{Aut}(\SO(n))$, permuting the position of the diagonal blocks $\SO(n_j)$, $j=1,\dots,s$, in the embedding of $H$ in $G$, so that 
\begin{equation}\label{assum}n_j\neq 2 \ \ \makebox{for} \ \ j=0,\dots,r \ \ \makebox{and} \ \ n_{r+1}=n_{r+2}=\cdots =n_s=2, \ \ \makebox{i.e.} \end{equation}
\begin{equation*} \fr{n}_{\fr{g}}(\fr{h})=\fr{so}(n_0)\times \cdots \times \fr{so}(n_{r})\times \underbrace{\fr{so}(2)\times \cdots \times \fr{so}(2)}_{s-r \ \makebox{times}}, \ \ r=0,\dots, s.\end{equation*}
For example, any diagonal embedding of $H=\SO(2)\times \SO(3)\times \SO(3)$ in $\SO(10)$ is conjugate to 

\smallskip
$\begin{pmatrix}
\SO(3) & 0 & 0&0\\
0& \SO(3) & 0 & 0\\
0& 0 & \SO(2) & 0\\
0 &0 &0 & \op{Id}_{2}
\end{pmatrix}$, in which case $\fr{n}_{\fr{g}}(\fr{h})=\underbrace{\fr{so}(3)\oplus \fr{so}(3)\oplus \fr{so}(2)}_{\fr{h}}\oplus \underbrace{\fr{so}(2)}_{\fr{n}}$. 

\smallskip
\noindent Such a conjugation leaves the space $\fr{p}$ invariant by permuting the $\op{ad}(\fr{n}_{\fr{g}}(\fr{h}))$-submodules $\fr{m}_{ij}$, $0\leq i<j\leq s$. Under this conjugation and in the case where the number $s-r$ of $\SO(2)$-blocks is greater than $2$, relation \eqref{Condition1} becomes
\begin{equation}\label{Condition2}
\left.A\right|_{\fr{m}_{ij}}=\left\{ 
\begin{array}{lll}\lambda_{ij}\op{Id},  \ \ \makebox{if $0\leq i<j\leq r+1$}\\ \\
\begin{pmatrix} \lambda_{ij}^1\left.\op{Id}\right|_{V_{ij}^1} & 0\\
0& \lambda_{ij}^2\left.\op{Id}\right|_{V_{ij}^2}\end{pmatrix}, \ \ \makebox{if $r+1\leq i<j\leq s$}.\end{array}
\right.   
\end{equation} 
 If $s-r<2$ then relation \eqref{psad} is trivially true.  
 
 We now assume that $s-r\geq 2$ and consider the space $\widetilde{M}:=\SO(2(s-r))/\underbrace{\SO(2)\times \cdots \times \SO(2)}_{(s-r)-\makebox{times}}$, whose tangent space is $\widetilde{\fr{m}}:=\bigoplus_{r+1\leq i<j\leq s}\fr{m}_{ij}\subset \fr{p}$. By Lemma \ref{restr} the restriction $\left.A\right|_{\widetilde{\fr{m}}}$ defines a $\SO(2(s-r))$-invariant g.o. metric on $\widetilde{M}$. 
 
We consider the following cases: 
\begin{equation*} \makebox{{\bf Case I}}  \ \ s-r>2 \ \ \makebox{and \ {\bf Case II}} \ \ s-r=2.\end{equation*}
 
 In {\bf Case I}, it is  $\SO(2(s-r))\neq \SO(4)$ and hence $\SO(2(s-r))$ is simple. On the other hand, the isotropy subgroup $\underbrace{\SO(2)\times \cdots \times \SO(2)}_{(s-r)-\makebox{times}}$ of $\widetilde{M}$ is abelian.  By the main theorem in \cite{So2}, the g.o. metric $\left.A\right|_{\widetilde{\fr{m}}}$ is standard and thus $\left.A\right|_{\widetilde{\fr{m}}}=\lambda\left.\op{Id}\right|_{\widetilde{\fr{m}}}$.  In particular, the last relation implies that the quantities $\lambda_{ij}^1$, $\lambda_{ij}^2$ in relation \eqref{Condition1} are equal for all $i,j$ with $r+1\leq i<j\leq s$, and thus relation \eqref{psad} is true for {\bf Case I}.
 
 In {\bf Case II}, we have $\widetilde{M}=\SO(4)/\SO(2)\times \SO(2)$, $\widetilde{\fr{m}}=\fr{m}_{s-1,s}=V_{s-1,s}^1\oplus V_{s-1,s}^2$ and relation \eqref{Condition2} yields
  $\left.A\right|_{\widetilde{\fr{m}}}=\begin{pmatrix} \lambda_{s-1,s}^1\left.\op{Id}\right|_{V_{s-1,s}^1} & 0\\
0& \lambda_{s-1,s}^2\left.\op{Id}\right|_{V_{s-1,s}^2}\end{pmatrix}$.  We consider the $\op{ad}(\fr{n}_{\fr{g}}(\fr{h}))$-irreducible submodules $\fr{m}_{1,s-1}$, and $V_{s-1,s}^1$. 

\noindent Relation \eqref{SubmoduleBrackets} yields $[\fr{m}_{1,s-1},V_{s-1,s}^1]\subseteq [\fr{m}_{1,s-1},\fr{m}_{s-1,s}]=\fr{m}_{1,s}$. More importantly, along with the description of the submodules $V_{ij}^l$ in Remark \ref{SO(4)2}, we deduce that $[\fr{m}_{1,s-1},V_{s-1,s}^1]\subseteq \fr{m}_{1,s}\setminus \{0\}$.  As a result, the space $[\fr{m}_{1,s-1},V_{s-1,s}^1]$ has non zero projection on $(\fr{m}_{1,s-1}\oplus V_{s-1,s}^1)^{\bot}$.  By using the first part of Lemma \ref{EigenEq}, the last relation, along with the facts that $\left.A\right|_{\fr{m}_{1,s-1}}=\lambda_{1,s-1}\left.\op{Id}\right|_{\fr{m}_{1,s-1}}$ and $\left.A\right|_{V_{s-1,s}^1}=\lambda_{s-1,s}^1\left.\op{Id}\right|_{V_{s-1,s}^1}$, yield 
\begin{equation}\label{geom1}\lambda_{1,s-1}=\lambda_{s-1,s}^1.
\end{equation}
\noindent By using the same argument for $V_{s-1,s}^2$ we deduce that
  \begin{equation}\label{geom2}\lambda_{1,s-1}=\lambda_{s-1,s}^2\end{equation}
\noindent Equations \eqref{geom1} and \eqref{geom2} yield $\lambda_{s-1,s}^1=\lambda_{s-1,s}^2$ and thus relation \eqref{psad} is also true for {\bf Case II}.  This concludes \textbf{Step 1.}

For \textbf{Step 2} we proceed as follows: Due to relation \eqref{SubmoduleBrackets}, along with Equation \eqref{psad}, the $\op{ad}({\fr{h}})$-invariance of $\fr{m}_{ij},\fr{m}_{jl}$ and their $B$-orthogonality, part \textbf{2.} of Lemma \ref{EigenEq} yields 
\begin{equation*}\label{EqEigen1}\lambda_{ij}=\lambda_{jl}=\lambda_{il} \ \ \makebox{for all} \ \ 0\leq i<j< l\leq s.\end{equation*}
 \noindent By Lemma \ref{CombinatorialLemma1} we deduce that the set $R_s=\{\lambda_{ij}: \ 0\leq i<j\leq s\}$ of the eigenvalues of $\left.A\right|_{\fr{p}}$ has only one element, and thus \textbf{Step 2} is concluded.
 \qed 

\medskip
 Before we proceed to the proof of the second proposition, we note that if $n_0=0$ or $n_0=1$, then $\fr{n}=\{0\}$ and $\fr{m}=\fr{p}$, and thus Theorem \ref{main1} follows directly from Proposition \ref{Simplif1}.\\

\noindent \emph{Proof of Proposition \ref{Simplif2}}.  We recall the spaces $\fr{m}_{0j}=\fr{m}^j_1\oplus \cdots \oplus \fr{m}^j_{n_0}$ defined in Section \ref{isotropy1}.  
By Proposition \ref{Simplif1} we have
\begin{equation}\label{Res2}\left.A\right|_{\fr{m}^j_i}=\lambda\op{Id}, \ \ i=1,\dots, n_0.\end{equation}
The Lie algebra $\fr{n}=\fr{so}(n_0)$ coincides with the Lie algebra of $N_G(H^0)/H^0=N_G(H)/H$.  By Lemma \ref{DualNormalizer}, $\left.A\right|_{\fr{n}}$ defines a bi-invariant metric on $N_G(H)/H$, which in turn corresponds to an $\op{Ad}$-invariant inner product on $\fr{so}(n_0)$.  For $n_0=2$, $\fr{n}$ is one-dimensional and thus $\left.A\right|_{\fr{n}}=\mu\op{Id}$.  For $2< n_0\neq 4$, $\fr{so}(n_0)$ is simple and the only $\op{Ad}$-invariant inner product is a scalar multiple of the Killing form.  Therefore, if $2< n_0\neq 4$ we also have $\left.A\right|_{\fr{n}}=\mu\op{Id}$.  For both cases, choose the vectors $e_{12}\in \fr{n}$ and $e_{1,n_0+1}\in \fr{m}^1_1$.  
We have 
\begin{equation*}[e_{12},e_{1n_0+1}]=-e_{2,n_0+1}\in\fr{m}^1_2.\end{equation*}
 Therefore, $[\fr{n},\fr{m}_1^1]$ has non-zero projection on $(\fr{n}\oplus \fr{m}_1^1)^{\bot}$.  Along with the $\op{ad}(\fr{h})$-invariance of $\fr{n}$ and $\fr{m}_1^1$ and the facts that $\left.A\right|_{\fr{m}^j_i}=\lambda\op{Id}$ and $\left.A\right|_{\fr{n}}=\mu\op{Id}$, Lemma \ref{EigenEq} yields $\lambda=\mu$.  We conclude that $\left.A\right|_{\fr{so}(n_0)}=\lambda\op{Id}$ if $2\leq n_0\neq 4$.

If $n_0=4$, $\fr{n}$ decomposes into two simple ideals as 
\begin{equation*}\fr{n}=\fr{so}(4)=\fr{n}_1\oplus \fr{n}_2,\end{equation*}
 where $\fr{n}_1=\op{span}\{e_{12}+e_{34},-e_{13}+e_{24},e_{23}+e_{14}\}\approx \fr{so}(3)$ and $\fr{n}_2=\op{span}\{e_{12}-e_{34},-e_{13}-e_{24},e_{23}-e_{14}\}\approx \fr{so}(3)$.  
 By Lemma \ref{GOLieGroups} we have 
\begin{equation}\label{Res1}\left.A\right|_{\fr{n}_1}=\mu_1\op{Id}\ \ \makebox{and} \ \ \left.A\right|_{\fr{n}_2}=\mu_2\op{Id}.
\end{equation}
 It remains to show that $\mu_1=\mu_2=\lambda$.  To this end, choose the vectors $e_{12}+e_{34}\in \fr{n}_1$, $e_{12}-e_{34}\in \fr{n}_2$ and $e_{15}\in \fr{m}^1_1$. The submodules $\fr{n}_1,\fr{n}_2$ and $\fr{m}^1_1$ are $\op{ad}(\fr{h})$-invariant. Moreover, we have 
 \begin{equation*}\label{Res3}[e_{12}+e_{34}, e_{15}]=- e_{25} \in \fr{m}^1_2 \subset (\fr{n}_1\oplus \fr{m}^1_1)^{\bot}\ \ \makebox{and} \ \ [e_{12}-e_{34}, e_{15}]=- e_{25} \in \fr{m}^1_2 \subset(\fr{n}_2\oplus \fr{m}^1_1)^{\bot}.\end{equation*}
  Hence, $[\fr{n}_1,\fr{m}^1_1]$ and $[\fr{n}_2,\fr{m}^1_1]$ have non zero projections on $(\fr{n}_1\oplus \fr{m}_1^1)^{\bot}$ and $(\fr{n}_2\oplus \fr{m}^1_1)^{\bot}$ respectively.  Along with Equations \eqref{Res2} and \eqref{Res1}, part \textbf{1.} of Lemma \ref{EigenEq} yields $\mu_1=\lambda=\mu_2$, which concludes the proof.  \qed

\section{The space $M=G/H=\U(n)/\U(n_1)\times\cdots\times \U(n_s)$}
 
 \subsection{Isotropy representation of $G/H=\U(n)/\U(n_1)\times \cdots \times \U(n_s)$}\label{isotropy2}

Denote by $\mu_n$ the standard representation of $\U(n)$ in $\mathbb C^n$.  Then the complexified adjoint representation of $\U(n)$ is
\begin{equation}\label{st}\operatorname{Ad}^{\U(n)}\otimes \mathbb C=\mu_n \otimes_{\mathbb C} \bar{\mu}_n.\end{equation}
By taking into account the assumption that $H$ is embedded diagonally in $G$, we can identify $H$ with the subgroup
$$
H=\begin{pmatrix}
\operatorname{Id}_{n_0} &  & & 0\\
   & \U(n_1) &  & \vdots \\
   &  & \ddots & \\
0   & \cdots &  & \U(n_s)\\
\end{pmatrix},
$$

\noindent of $G$, where $n_0:=n-(n_1+\cdots +n_s)$.  

Let $\tau _{n_i}: \U(n_{1})\times\cdots\times\U(n_{s})\to\U(n_{i})$ be the projection onto the $i$-factor and $q_{i} = \mu_{n_{i}}\circ\tau_{n_{i}}$
 be the standard representation of $H$, i.e.  
$$ 
\U(n_{1})\times\cdots\times\U(n_{s})\stackrel{\tau_{n_{i}}}{\longrightarrow}\U(n_{i})\stackrel{\mu_{n_{i}}}{\longrightarrow}\Aut(\bb{C}^{n_{i}}). 
$$

By using relation \eqref{st}, we obtain 
\begin{eqnarray}
\left.\operatorname{Ad}^{G}\otimes \mathbb C\right|_H &=&\left.\mu_n \otimes_{\mathbb C} \bar{\mu}_n \right|_H=\left.\mu_n\right|_H\otimes_{\mathbb C} \left.\bar{\mu}_n\right|_H 
=
({q_{1}\oplus\cdots \oplus q_s}\oplus \mathbbm{1}_{n_0})\otimes_{\mathbb C}({\bar{q}_{1}}\oplus\cdots\oplus {\bar{q}_{s}}\oplus\mathbbm{1}_{n_0})
\nonumber \\
&=&
\mathbbm{1}_{{n_0}^2}\ \bigoplus_{i=1}^s\left\{ {q_{i}\otimes_{\mathbb C}\bar{q}_{i}} \right\} \bigoplus_{j=1}^s \{(q_j\oplus\bar{q}_j)\otimes_{\mathbb C}\mathbbm{1}_{n_0}\}\nonumber \\
&&
\bigoplus_{1\le i<j\le s}\left\{(q_{i}\otimes_{\mathbb C}\bar{q}_{j})\oplus (q_{j}\otimes_{\mathbb C}\bar{q}_{i})\right\}.
\label{rep2} 
\end{eqnarray}
The summand $\bigoplus_{i=1}^s\{{q_{i}\otimes_{\mathbb C}\bar{q}_{i}}\}$ in \eqref{rep2} corresponds to $\operatorname{Ad}^H\otimes \mathbb C$, therefore, by virtue of Proposition  \ref{isotrepr}, 
$\chi\otimes \mathbb C$ is given by 
\begin{eqnarray}
\chi\otimes \mathbb C&=&\mathbbm{1}_{n_0^2}\ 
\bigoplus_{j=1}^s \{(q_j\oplus\bar{q}_j)\otimes_{\mathbb C}\mathbbm{1}_{n_0}\}
\bigoplus_{1\le i<j\le s}\left\{(q_{i}\otimes_{\mathbb C}\bar{q}_{j})\oplus (q_{j}\otimes_{\mathbb C}\bar{q}_{i})\right\}.
\label{chi2} 
\end{eqnarray}
Expression \eqref{chi2} induces a {\it real} decomposition of the tangent space 
\begin{equation}\label{dd12}
\fr{m}=\fr{n}\oplus \fr{p},
\end{equation}
where
\begin{equation}\label{dd2}
  \fr{n}=\fr{u}(n_0),\ \ \ 
  \fr{p}= \bigoplus_{j=1}^s{\fr{m}_{0j}}\bigoplus_{1\leq i<j\leq s}\fr{m}_{ij}.
 \end{equation}
In the above decomposition we have 
$\fr{m}_{0j}=\fr{m}_1^j\oplus\fr{m}_2^j\oplus\cdots\oplus\fr{m}_{n_0}^j$ with
$\fr{m}_\alpha^j\cong\fr{m}_\beta^j$, $\alpha\ne\beta$ and $\dim(\fr{m}_\ell ^j)=2n_j, \ell =1,2,\dots ,n_0$.

We now give explicit matrix representations of the modules $\fr{m}_{0j}$ and $\fr{m}_{ij}$.
We consider the $\operatorname{Ad}(\U(n))$-invariant inner product $B:\fr{u}(n)\times \fr{u}(n)\rightarrow \mathbb R$ given by
\begin{equation}\label{inner2}B(X,Y)=-\operatorname{Trace}(XY), \quad X,Y\in \fr{u}(n).
\end{equation}
Then there is a $B$-orthogonal decomposition 
\begin{equation}\label{obt2}\fr{g}=\fr{h}\oplus \fr{m},
\end{equation}
 where $\fr{h}=\fr{u}(n_1)\oplus \cdots \oplus \fr{u}(n_s)$ and $\fr{m}\approx T_o(G/H)$.  
We consider a basis of $\fr{g}=\fr{u}(n)$ as follows:\\
Let $M_{n}\mathbb C$ be the set of complex $n\times n$ matrices and let $E_{ab}\in M_{n}\mathbb C$, $a,b=1,\dots,n$, be the matrix with  1 in the $(a,b)$-entry and zero elsewhere.  For $a,b=1,\dots,n$, we set
\begin{equation}\label{mel2}e_{ab}=E_{ab}-E_{ba},\quad f_{ab}=\sqrt{-1}(E_{ab}+E_{ba}).
\end{equation} 
Note that
$e_{ab}=-e_{ba}, {f}_{ab}={f}_{ba}$.
 Then the set 
\begin{equation}\label{set2}\mathcal{B}=\left\{ e_{ab}, {f}_{cd}:1\le a<b\le n ,\  1\le c \le d\le n \right\},
\end{equation}  
 constitutes a basis of $\fr{u}(n)$ which is orthogonal with respect to $B$.  We have the following.

\begin{lemma}\label{rel2}
The non zero bracket relations among the vectors \eqref{mel2} are given by
\begin{equation*}
\begin{array}{ccc}
 [e_{ab},e_{cd}]=\delta_{bc}e_{ad}-\delta_{ad}e_{cb}-\delta_{ac}e_{bd}-\delta_{bd}e_{ac},&
[{f}_{ab},e_{cd}]=\delta_{bc}{f}_{ad}-\delta_{ad}{f}_{cb}+\delta_{ac}{f}_{bc}-\delta_{bd}f_{ac},\\

\ \ \ [{f}_{ab}, {f}_{cd}]=-\delta_{bc}e_{ad}+\delta_{ad}e_{cb}-\delta_{ac}e_{bd}-\delta_{bd}e_{ac}.\end{array}
\end{equation*}
\end{lemma}

\begin{proof}We observe that $[E_{ab},E_{cd}]=\delta_{bc}E_{ad}-\delta_{ad}E_{cb}$ and the lemma follows by direct computation.
\end{proof}

Then a choice of the modules in the decomposition \eqref{dd2} is the following:
\begin{eqnarray*}
\fr{m}_{0j}&=&\operatorname{span}\left\{ {e_{ab}, f_{cd}\in \mathcal{B}: 1 \leq a,c \leq n_0, \ n_0+n_1+\cdots+n_{j-1}+1\leq b,d \leq n_0+n_1+\cdots+n_{j}} \right\},\nonumber\\
\fr{m}_{ij}&=&\operatorname{span} \{ {e_{ab}, f_{cd}\in \mathcal{B}:n_0+n_1+\cdots+n_{i-1}+1\leq a,c \leq n_0+n_1+\cdots+n_i,} \nonumber
\\
&& n_0+n_1+\cdots+n_{j-1}+1\leq b,d \leq n_0+n_1+\cdots+n_j \}, \quad 1\leq i<j\leq s,\\
\fr{u}(n_0)&=&\Span\left\{e_{ab}, {f}_{cd}\in \mathcal{B}: 1\le a<b=n, \ 1\le c\le d\le n\right\}.
\end{eqnarray*}
The equivalent modules in the decomposition of $\fr{m}_{0j}$ are given by
$$
\fr{m}^j_\ell =\Span\{e_{\ell b}, f_{\ell d}\in \mathcal{B}: n_0+n_1+\cdots +n_{j-1}+1\le b,d\le 
n_0+n_1+\cdots +n_j \}, \ \ell = 1, \dots , n_0.
$$
Also,
$$
\fr{u}(n_j)=\Span\{e_{ab}, f_{cd}\in\mathcal{B}: n_0+n_1+\cdots +n_{j-1}+1\le a,b,c,d\le n_0+n_1+\cdots +n_j\}, \quad j =1, \dots , s.
$$
Then the $B$-orthogonal decompositions \eqref{dd12} and \eqref{dd2} can be depicted in the following matrix, which shows the upper triangular part of $\fr{u}(n)$:
$$
\begin{pmatrix}
\fr{u}(n_0) & \fr{m}_{01} & \fr{m}_{02} & \fr{m}_{03} & \cdots & \fr{m}_{0s}\\
        &  0_{n_1} & \fr{m}_{12} & \fr{m}_{13} & \cdots & \fr{m}_{1s}\\
    &   & 0_{n_2}  & \fr{m}_{23} & \cdots & \fr{m}_{2s}\\
   & \ast  &      & \ddots &  & \vdots\\
   &   &      &  &  & 0_{n_s}     
\end{pmatrix}
$$
The matrices $\fr{m}_{0j}$ are of size $n_0\times n_j$, the matrices $\fr{m}_{ij}$ are of size $n_i\times n_j$, and the matrices $\fr{u}(n_i)$ of size $n_i\times n_i$.

\subsection{Proof of Theorem \ref{main2}}\label{proof2}

Recall the spaces $\fr{n}$ and $\fr{p}$ defined in subsection \ref{isotropy2} and the decompositions \eqref{dd12} and \eqref{dd2}.   
We note that if $n_0=0$ (i.e. $n_1+\cdots +n_s=n$), then $\fr{n}=\{0\}$, and Theorem \ref{main2} follows directly from Proposition \ref{Simplif12} below. 
Assume henceforth that $n_0\geq 1$.
To prove Theorem \ref{main2} we need the following two propositions, which we  will prove at the end of the subsection.

\begin{prop}\label{Simplif12}
$\left.A\right|_{\fr{p}}=\lambda\op{Id}$, where $\lambda>0$. 
\end{prop}

   We further decompose $\fr{n}=\fr{u}(1)\oplus \fr{su}(n_0)=\fr{z}(\fr{n})\oplus \fr{su}(n_0)$ into its simple and abelian ideals.  With respect to the basis $\mathcal{B}$, we have $\fr{z}(\fr{n})=\op{span}\{f_{11}+\cdots +f_{n_0n_0}\}$ and $\fr{su}(n_0)$ is the $B$-orthogonal complement of $\fr{z}(\fr{n})$ in $\fr{u}(n_0)$.

\begin{prop}\label{Simplif22}
Let $\fr{n}=\fr{u}(1)\oplus \fr{su}(n_0)=\fr{z}(\fr{n})\oplus \fr{su}(n_0)$ be the decomposition of $\fr{n}=\fr{u}(n_0)$ into its simple and abelian ideals. If $n_0=1$ then $\left.A\right|_{\fr{n}}=\left.\mu\op{Id}\right|_{\fr{n}}$ for some $\mu>0$. If $n_0\geq 2$, then $\left.A\right|_{\fr{n}}=\begin{pmatrix} \left.\mu\op{Id}\right|_{\fr{z}(\fr{n})} & 0\\
0 & \left.\lambda\op{Id}\right|_{\fr{su}(n_0)}\end{pmatrix}$, where $\lambda$ is given by {\rm Proposition \ref{Simplif12}}.
\end{prop}

 From Propositions \ref{Simplif12} and \ref{Simplif22}, the decomposition $\fr{m}=\fr{n}\oplus \fr{p}=\fr{u}(1)\oplus \fr{su}(n_0)\oplus \fr{p}$, and after normalizing the metric, we conclude that any g.o. metric on $G/H$ has necessarily the form (up to homothety)
\begin{equation}\label{fform}A=\begin{pmatrix} \left.\mu\op{Id}\right|_{\fr{z}(\fr{n})} & 0\\
0 & \left.\op{Id}\right|_{\fr{su}(n_0)\oplus \fr{p}}\end{pmatrix}.\end{equation}
 To conclude the proof of Theorem \ref{main2}, it remains to prove that the above form is also sufficient, i.e. the metrics $A$ are g.o. metrics.  Let $X\in \fr{m}$.  By Proposition \ref{GOCond}, we need to find a vector $a\in \fr{h}$ such that 
\begin{equation}\label{ggo}[a+X,AX]=0.\end{equation}
Let $X_{\fr{p}}$, $X_{\fr{su}(n_0)}$ and $X_{\fr{z}(\fr{n})}$ denote the projections of $X$ on $\fr{p}$, $\fr{su}(n_0)$ and $\fr{z}(\fr{n})$ respectively.  Then 
\begin{equation*}\label{a}X_{\fr{z}(\fr{n})}=r\sum_{i=1}^{n_0}{f_{ii}},\quad \mbox{for some} \ r\in \mathbb{R},\end{equation*}
 and $X_{\fr{p}}=X_1+X_2$, where $X_1$ is the projection of $X_{\fr{p}}$ on $\fr{m}_{01}\oplus \cdots \oplus \fr{m}_{0s}$ and $X_2$ is the projection of $X_{\fr{p}}$ on $\bigoplus_{1\leq i<j\leq s}\fr{m}_{ij}$.  Moreover, we write
\begin{equation*}X_1=\sum_{i=1}^{n_0}\sum_{j=n_0+1}^n{(a_{ij}e_{ij}+b_{ij}f_{ij})}, \quad a_{ij},b_{ij}\in \mathbb R.\end{equation*}
 We also consider the vector
\begin{equation*}\widetilde{X}_{1}:=\sum_{i=1}^{n_0}\sum_{j=n_0+1}^n{(b_{ij}e_{ij}-a_{ij}f_{ij})}.\end{equation*} 
 Finally, we choose the vector
\begin{equation*}a=r(1-\mu)\sum_{i=n_0+1}^n{f_{ii}} \in \fr{h}=\fr{u}(n_1)\oplus \cdots \oplus \fr{u}(n_s).\end{equation*} 
 By using Lemma \ref{rel2} it is straightforward to check the following relations:
\begin{equation}\label{sf} [X_{\fr{z}(\fr{n})},X_{1}]=-2r\widetilde{X}_{1}, \ \ [X_{\fr{z}(\fr{n})},X_2]=0, \quad [a,X_{1}]=2r(1-\mu)\widetilde{X}_{1}\ \ \makebox{and} \ \ [a,X_2]=0.\end{equation}
  More specifically, the last relation can be verified by viewing both vectors $a,X_2$ as elements of the Lie algebra $\fr{k}:=\fr{u}(n_1+\cdots +n_s)$, embedded diagonally in $\fr{g}$ as $\begin{pmatrix} 0_{n_0\times n_0} & 0\\ 0 & \fr{k}\end{pmatrix}$, and observing that $a$ lies in the center of $\fr{k}$.  
  Finally, since $[\fr{h},\fr{n}]=0$ and $[\fr{z}(\fr{n}),\fr{n}]=0$, we obtain  
\begin{equation}\label{sf1}[a,X_{\fr{z}(\fr{n})}]=[a,X_{\fr{su}(n_0)}]=[X_{\fr{z}(\fr{n})},X_{\fr{su}(n_0)}]=0.\end{equation}
 We can now  verify condition \eqref{ggo}. Indeed, by taking into account relations \eqref{sf} and \eqref{sf1}, as well as Equation \eqref{fform}, we obtain 
\begin{eqnarray*}[a+X,AX]&=&[a+X_{\fr{p}}+X_{\fr{su}(n_0)}+X_{\fr{z}(\fr{n})}, X_{\fr{p}}+X_{\fr{su}(n_0)}+\mu X_{\fr{z}(\fr{n})}]\\ \nonumber
&=&
[a,X_{\fr{p}}]+(1-\mu)[X_{\fr{z}(\fr{n})}, X_{\fr{p}}]\nonumber \\
&=&
[a,X_1+X_2]+(1-\mu)[X_{\fr{z}(\fr{n})}, X_1+X_2]\nonumber\\
&=&
[a,X_1]+(1-\mu)[X_{\fr{z}(\fr{n})}, X_1]\nonumber\\
&=&
2r(1-\mu)\widetilde{X}_{1}-2r(1-\mu)\widetilde{X}_{1}=0,\end{eqnarray*} 
and this concludes the proof of the theorem.\qed

\medskip
We now give the proofs of the above propositions.

\smallskip
\emph{Proof of Proposition \ref{Simplif12}}.  Since $H$ is connected, the Lie algebra of the normalizer $N_G(H)$ coincides with $\fr{n}_{\fr{g}}(\fr{h})=\{Y\in \fr{g}: [Y,\fr{h}]\subseteq \fr{h}\}$.  In our case, $\fr{h}=\fr{u}(n_1)\oplus \cdots \oplus \fr{u}(n_s)$ and 
\begin{equation*}\fr{n}_{\fr{g}}(\fr{h})=\fr{u}(n_0)\oplus \fr{u}(n_1)\oplus \cdots \oplus \fr{u}(n_s)=\fr{n}\oplus \fr{h}.
\end{equation*}
 Therefore, the tangent space of $G/N_G(H)$ coincides with $\fr{p}$.  The normalizer Lemma \ref{NormalizerLemma} then implies that $\left.A\right|_{\fr{p}}$ defines a $G$-invariant g.o. metric on $G/N_G(H)$. 

 Similarly to the space $\SO(n)/\SO(n_1)\times \cdots \times \SO(n_s)$, by taking into account Lemma \ref{rel2} and the expressions of the subspaces $\fr{m}_{ij}, \fr{u}(n_j), \fr{u}(n_0)$ in terms of the basis $\mathcal{B}$, we deduce that the submodules $\fr{m}_{ij}$, $0\leq i<j\leq s$ are $\op{ad}(\fr{n}_{\fr{g}}(\fr{h}))$-invariant, $\op{ad}(\fr{n}_{\fr{g}}(\fr{h}))$-irreducible and pairwise inequivalent.  Therefore, 
 \begin{equation}\label{9gel} \left.A\right|_{\fr{m}_{ij}}=\lambda_{ij}\op{Id}.\end{equation}
 It remains to prove that all $\lambda_{ij}$ are equal.  To this end, we will use similar arguments as in \textbf{Step 2} in Proposition \ref{Simplif1}.  More specifically, Lemma \ref{rel2} yields 
\begin{equation*}\label{SubmoduleBrackets1}
[\fr{m}_{ij},\fr{m}_{jl}]=\fr{m}_{il} \ \ \makebox{for all} \ \ 0\leq i<j<l\leq s.
\end{equation*}
Using the above relation, along with Equation \eqref{9gel}, the $\op{ad}({\fr{h}})$-invariance of $\fr{m}_{ij},\fr{m}_{jl}$ and their $B$-orthogonality, part \textbf{2.} of Lemma \ref{EigenEq} yields that 
\begin{equation*}\label{EqEigen1}\lambda_{ij}=\lambda_{jl}=\lambda_{il} \ \ \makebox{for all} \ \ 0\leq i<j< l\leq s.\end{equation*}
 By Lemma \ref{CombinatorialLemma1}, we deduce that the set $R_s=\{\lambda_{ij}: \ 0\leq i<j\leq s\}$ of the eigenvalues of $\left.A\right|_{\fr{p}}$ has only one element, therefore $\left.A\right|_{\fr{p}}=\lambda\op{Id}$.\qed 

\medskip
\noindent \emph{Proof of Proposition \ref{Simplif22}}.  The Lie algebra $\fr{n}=\fr{u}(n_0)$ coincides with the Lie algebra of $N_G(H)/H$.  By Lemma \ref{DualNormalizer}, $\left.A\right|_{\fr{n}}$ defines a bi-invariant (and hence g.o.) metric on $U(n_0)$, which in turn corresponds to an $\op{Ad}$-invariant inner product on $\fr{u}(n_0)$.  Since the center $\fr{z}(\fr{n})$ of $\fr{n}$ is one-dimensional, Lemma \ref{GOLieGroups} yields 

\begin{equation*}\label{niad}\left.A\right|_{\fr{n}}=\begin{pmatrix} \left.\mu\op{Id}\right|_{\fr{z}(\fr{n})} & 0\\
0 & \left.\widetilde{\lambda}\op{Id}\right|_{\fr{su}(n_0)}\end{pmatrix}, \ \ \widetilde{\lambda}>0.\end{equation*}

If $n_0=1$, then $\fr{n}=\fr{u}(1)=\fr{z}(\fr{n})$, verifying Proposition \ref{Simplif22} for this case.  Assume that $n_0\geq 2$. It remains to show that $\widetilde{\lambda}$ is equal to the eigenvalue $\lambda$ given in Proposition \ref{Simplif12}. We recall the spaces $\fr{m}_{0j}=\fr{m}^j_1\oplus \cdots \oplus \fr{m}^j_{n_0}$ defined in Section \ref{isotropy2}.  By Proposition \ref{Simplif12}, we have
\begin{equation*}\label{Res22}\left.A\right|_{\fr{m}^j_i}=\lambda\op{Id}.\end{equation*}
 Choose the vectors $e_{12}\in \fr{su}(n_0)$ and $e_{1,n_0+1}\in \fr{m}^1_1$.  By Lemma \ref{rel2}, we have 
\begin{equation*}[e_{12},e_{1,n_0+1}]=-e_{2,n_0+1}\in\fr{m}^1_2.\end{equation*}
 Therefore, $[\fr{su}(n_0),\fr{m}_1^1]$ has non-zero projection on $(\fr{su}(n_0)\oplus \fr{m}_1^1)^{\bot}$.  Along with the $\op{ad}(\fr{h})$-invariance of $\fr{su}(n_0)$ and $\fr{m}_1^1$ and the facts that $\left.A\right|_{\fr{m}^j_i}=\lambda\op{Id}$ and $\left.A\right|_{\fr{su}(n_0)}=\widetilde{\lambda}\op{Id}$, Lemma \ref{EigenEq} yields $\lambda=\widetilde{\lambda}$. \qed \\

\end{document}